\documentclass[12pt, twoside, a4paper]{article}

\usepackage{amssymb}
\usepackage{geometry}
\usepackage[T1]{fontenc}
\usepackage[ngerman,english]{babel}
\usepackage[latin1]{inputenc}

\usepackage[leqno]{amsmath}
\usepackage{amsthm}
\usepackage{amsfonts}
\usepackage{leftidx}

\usepackage{dsfont}
\usepackage{marvosym} 
\usepackage{ stmaryrd } 
\usepackage{lmodern} 
\usepackage{emptypage}
\usepackage{pdfpages}

\usepackage{graphicx}
\usepackage{color}
\usepackage[bottom,hang]{footmisc} 
\newtheorem{Def}{Definition}[section] 

\newtheorem{thm}[Def]{Theorem}
\newtheorem{prop}[Def]{Proposition}
\newtheorem{lem}[Def]{Lemma}
\newtheorem{kor}[Def]{Corollary}

\newtheorem*{thmA}{Theorem}

\newtheorem*{propnonumber}{Proposition}

\newcommand{\Hess}{\mathrm{Hess}}

\newcommand{\grad}{\mathrm{grad}}

\newcommand{\Ric}{\mathrm{Ric}}
\newcommand{\ric}{\mathrm{ric}}
\newcommand{\scal}{\mathrm{scal}}
\newcommand{\tr}{\mathrm{tr}}

\newcommand{\A}{\mathcal{A}_n}

\newcommand{\vol}{\mathrm{vol}}
\newcommand{\R}{\mathbb{R}}

\renewcommand{\phi}{\varphi}

\theoremstyle{definition} 

\newtheorem{deff}[Def]{Definition}

\newtheorem{bsp}[Def]{Example}
\newtheorem{bem}[Def]{Remark}

\usepackage[numbers,square]{natbib}
\title{Bianchi-convex sets and \\ a new maximum principle for the Ricci flow}
\author{Stine Franziska Beitz}
\date{}

\begin{document}

\maketitle

\begin{abstract}
We introduce the new notion of \textit{Bianchi-convex sets}, a generalization of convex sets of algebraic curvature tensors inspired by the second Bianchi identity. It turns out that Hamilton's maximum principle for the Ricci flow can be generalized for Bianchi-convex sets. 
\end{abstract}




\section*{Introduction} 
\addcontentsline{toc}{section}{Introduction}

The notion of convexity plays an important role in many areas of mathematics. However, in various situations where this concept is used, only a weaker form of convexity is actually needed. The purpose of this paper is to exploit such an observation in the case of the Ricci flow, where convexity is of particular interest via Hamilton's maximum principle.
\medskip

Recall that in Ricci flow, metrics evolve according to the partial differential equation
\begin{align} \label{RFequation}
\frac{\partial}{\partial t}g_t = -2\ric_{g_t},
\end{align}
where $\ric_{g_t}$ denotes the Ricci tensor of $g_t$. Richard S. Hamilton, who was the first to introduce the Ricci flow in 1982 \cite{3mfdsHamilton}, proved short-time existence and uniqueness of solutions to the Ricci flow to a given initial metric and used it to show that simply-connected 3-manifolds with positive Ricci curvature evolve to a round sphere under the flow. Later, the Ricci flow was the main tool in Perelman's proof of Thurston's geometrization conjecture for three-manifolds \cite{EntropyPerelman,RicciflowWithSurgery,FiniteExtinctionTime}, which in particular implies the Poincaré conjecture and Thurston's elliptization conjecture. 
\medskip

Hamilton's maximum principle states that an $O(n)$-invariant, closed and convex subset $\Omega$ of the space $\A$ of algebraic curvature tensors, which is invariant under the ordinary differential equation
\begin{align} \label{ODEeinleitung}
R'(t) = R(t)^2 + R(t)^{\#},
\end{align}
is already invariant under the Ricci flow \cite[Theorem 4.3]{4-mfdsHamilton}. Here, $\#$ is a certain $O(n)$-equivariant quadratic map. This theorem is paramount when trying to classify manifolds that satisfy certain curvature conditions, such as compact 4-manifolds with positive curvature operator \cite{4-mfdsHamilton}, the theorem of Böhm and Wilking for compact manifolds with 2-positive curvature operator \cite{BoehmWilkingAnnals} and the differentiable sphere theorem by Brendle and Schoen \cite{1/4pinchedcurvature}. 

\medskip

In the proof of Hamilton's maximum principle, a crucial step is to show that if the Riemannian curvature tensor $Rm$ of some $n$-dimensional manifold $M$ is contained\footnote{More precisely, we transfer $\Omega$ to the fibres of the bundle of algebraic curvature tensors $S^2_B(\Lambda^2T^*M)$, which makes sense due to the $O(n)$-invariance of $\Omega$.} in $\Omega$, then it satisfies  
\begin{align*}
\sum_{i=1}^n \text{\Gemini}^{\partial \Omega}_{Rm(x)}(\nabla_{e_i}Rm, \nabla_{e_i}Rm) \leq 0
\end{align*}
for an orthonormal basis $\{e_1, \dots, e_n\}$ of $T_xM$, whenever $x \in M$ with $Rm(x) \in \partial \Omega$. If $\Omega$ is convex, this follows from the fact that each summand is non-positive as in this case (in the convention used here) the second fundamental form $\text{\Gemini}^{\partial \Omega}$ of $\partial \Omega$ is negative semidefinite. The observation is now that it suffices to ensure the non-positivity of the sum instead of the individual summands. This motivates the definition of a  \textit{Bianchi-convex set}.
\medskip

We define a closed subset $\Omega \subseteq \A$ with smooth boundary to be \textit{Bianchi-convex}, if for all $R \in \partial \Omega$ and tuples $(T_1, \dots, T_n) \in (T_R\partial \Omega)^n$ which satisfy a certain algebraic second Bianchi identity, we have that
\begin{align*}
\sum_{i=1}^n \text{\Gemini}^{\partial \Omega}_R(T_i,T_i) \leq 0,
\end{align*}
where $\text{\Gemini}^{\partial \Omega}_R$ denotes the second fundamental form of $\partial\Omega$ in $R$. (In Def.~\ref{DefBianchiconvexSet} below, we generalize to sets with non-smooth boundary.) The main goal of this paper is a proof of the following generalization of Hamilton's maximum principle.

\begin{thmA} 
Let $\Omega \subseteq \A$ be $O(n)$-invariant, closed, Bianchi-convex and uniformly transversally star-shaped with respect to $\lambda I$ for some $\lambda \in \R$ (see Definition \ref{coneCondition}). If $\Omega$ is invariant under the ordinary differential equation \eqref{ODEeinleitung},
then $\Omega$ is invariant under the Ricci flow.
\end{thmA}

Here, $I$ denotes the identity in $\A$. The requirement that $\Omega$ is uniformly transversally star-shaped with respect to $\lambda I$ is a technical condition that is satisfied in all known cases where $\Omega$ is $O(n)$-invariant, closed, convex and invariant under \eqref{ODEeinleitung} (c.f.\ Remark \ref{HamiltonSetsUTS}). 

\medskip

In order to give applications for this theorem, one needs to gain some understanding of Bianchi-convex sets. In dimension three, this was achieved by the author in her thesis \cite{ThesisSFB}, at least for sets of the form 
\begin{align*}
\Omega_f := \{ R \in \mathcal{A}_3 \mid f(\lambda_1(R), \lambda_2(R), \lambda_3(R)) \leq 0 \}
\end{align*}
for a smooth symmetric function $f: \R^3 \rightarrow \R$ with $f^{-1}(0) \neq \emptyset$ such that 0 is a regular value of $f$ (here, $\lambda_1(R) \leq \lambda_2(R) \leq \lambda_3(R)$ denote the eigenvalues of $R$). We review these results in Section~\ref{BconvexSetsInDim3}. A particular result is the following example for a Bianchi-convex set which is not convex.

\begin{propnonumber}
For $a \in (\frac{1}{3}, \frac{2}{5})$ and $c>0$, there exists a constant $b_{a,c}>0$ such that the set
\begin{align*}
\big\{R \in \mathcal{A}_3 \ \big{|} \ \|R\|^2 - a\, \scal(R)^2 \leq c \ \ \text{and} \ \ \scal(R) \geq b_{a,c} \big\}
\end{align*}
is Bianchi-convex and invariant under \eqref{ODEeinleitung}, thus invariant under the Ricci flow. 
\end{propnonumber}

This paper is organized as follows. After discussing some preliminary notions, in Chapter 2 we introduce curvature conditions $\Omega$ and the corresponding subsets $\Omega^g$ of the bundle of algebraic curvature tensors over a Riemannian manifold $(M,g)$. Moreover, we give some properties of tangent cones we use later in the proof as well as a connection to subsets of a vector space which are invariant under an ordinary differential equation of the form $f'(t) = \Phi(f(t))$. In Chapter 3, we investigate the notion of Bianchi-convex sets and in dimension three give examples of such sets of algebraic curvature tensors that are not convex but invariant under \eqref{ODEeinleitung}. The fourth chapter is dedicated to a proof of Theorem A.  \\

\textbf{Acknowledgements.}
At this point, I would like to thank my advisor Burkhard Wilking for suggesting this interesting topic to me, for his guidance and many helpful discussions. 
I am also grateful to Christoph Böhm, Ramiro Lafuente, Jonas Stelzig, Artem Nepechiy and Ricardo Mendes for numerous inspiring conversations.
My special thanks go out to Matthias Ludewig for his continuing support and encouragement, careful proofreading and many useful proposals. 
Moreover, I  thank the SFB 878 and the MSRI for financial support, and the University of Adelaide for hospitality.

\section{Preliminaries}

This chapter is dedicated to introducing all known objects, spaces, notions and facts that will be important throughout the  present work. We define the space of algebraic curvature tensors, which plays a crucial role when investigating the curvature of a Riemannian manifold, and give some first properties. Moreover, we consider solutions to the Ricci flow and introduce a corresponding metric connection on space-time, which allows to formulate the evolution equation of the Riemannian curvature operator under this  flow.


\subsection[The space of algebraic curvature tensors]{The space of algebraic curvature tensors} \markboth{CHAPTER 1. \ PRELIMINARIES}{1.1. \ ALGEBRAIC CURVATURE TENSORS}
 \label{SecAlgCurvTensors}

Throughout this section, let $V$ be an $n$-dimensional Euclidean vector space. On the second exterior power $\Lambda^2V^*$ of its dual space $V^*$, i.e.\ on the space of antisymmetric bilinear forms on $V$, we consider the scalar product given in such a way that for each orthonormal basis $(b^1, \dots, b^n)$ of $V^*$ the vectors $b^i \wedge b^j$, $i < j$, form an orthonormal basis of $\Lambda^2V^*$. Using the scalar product on $V$, we will freely identify $\mathfrak{so}(V)$, the space of skew-symmetric endomorphisms on $V$, and $\Lambda^2V^*$ via the isomorphism
\begin{align} \begin{split} \label{isoLambdaSo}
\mathfrak{so}(V) &\longrightarrow \Lambda^2V^* \\
A &\longmapsto \omega_A,
\end{split}
\end{align}
where $\omega_A(v,w) := \langle  v, A(w) \rangle$ for $v,w \in V$.  Moreover, we choose the scalar product on $\mathfrak{so}(V)$ in such a way that the isomorphism \eqref{isoLambdaSo} is an isometry. This means that 
\begin{align*}
\langle A, B \rangle =  - \frac{1}{2} \tr(AB)
\end{align*}
for $A, B \in \mathfrak{so}(V)$. 
\medskip

By $S^2(V^*)$, we denote the \textit{space of symmetric bilinear forms on $V$} or equivalently, using the metric, \textit{the space of self-adjoint endomorphisms of $V$}. The space $S^2(\Lambda^2V^*)$ has the subspace $S^2_B(\Lambda^2V^*)$, the \textit{space of algebraic curvature tensors associated to $V$}, i.e.\ the space of symmetric bilinear forms $R$ on $\Lambda^2V$ which satisfy the first Bianchi identity, that is
\begin{align*}
R(x \wedge y, z \wedge w) + R(y \wedge z, x \wedge w) + R(z \wedge x, y \wedge w) = 0
\end{align*}
for all $x,y,z,w \in V$.
In the case that $V = \R^n$, we will denote $S^2_B(\Lambda^2(\R^n)^*)$ by $\mathcal{A}_n$. The scalar product on $\Lambda^2V^*$ induces a scalar product on $S^2(\Lambda^2V^*)$, and hence on $S^2_B(\Lambda^2V^*)$, given by  $\langle R,S \rangle = \tr(R\circ S)$, if $R$ and $S$ are considered as self-adjoint endomorphisms on $\Lambda^2V$. 
It is well-known that, if $V$ is of dimension $n\leq 3$, we have that $S^2_B(\Lambda^2V^*) = S^2(\Lambda^2V^*)$. 
\medskip

By looking at $S^2(\Lambda^2V^*)$ as a space of endomorphisms, there is the map $R \mapsto R^2$ sending an endomorphism to its square. It was noticed by Hamilton \cite{4-mfdsHamilton} that there is a second $O(n)$-equivariant quadratic map $\#$ from $S^2(\Lambda^2V^*)$ to itself. (Note that the explicit definition of $\#$ given in  \cite{4-mfdsHamilton} differs from ours by a factor $2$ since Hamilton uses a different scalar product on $\mathfrak{so}(V)$. For our purposes, we do not need the explicit formula anyway.) It is a fact that for an algebraic curvature tensor $R$, we have that $R^2 + R^{\#}$ is an algebraic curvature tensor as well. In other words, the sum $R^2+R^{\#}$ satisfies the first Bianchi identity even though the individual summands may not.
\medskip

Corresponding to an algebraic curvature tensor $R \in S^2_B(\Lambda^2V^*)$, we define the \textit{Ricci tensors} $\ric(R): V \times V \rightarrow \R$ and $\Ric(R): V \rightarrow \R$ by
\begin{align*}
\ric(R)(v, w) = \langle \Ric(R)(v), w \rangle = \frac{1}{2} \tr(R(v \wedge \cdot \;, w \wedge \cdot \;)),
\end{align*}
where $v, w \in V$, and the \textit{scalar curvature $\scal(R)$} by
\begin{align*}
\scal(R) := \tr(\Ric(R)).
\end{align*}

Notice, that here we use a different convention for the Ricci tensor and consequently the scalar curvature than e.g.\ in \cite{BoehmWilkingAnnals} and \cite{RFtechniquesCCGGCIIKLLL}, where the factor $\frac{1}{2}$ is omitted.

\begin{bem} \label{scal=2tr}
Let $(b_1, \dots, b_n)$ be an orthonormal basis of $V$. Then $(b_i \wedge b_j)_{i<j}$ is an orthonormal basis of $\Lambda^2V$, so that for $R \in S^2_B(\Lambda^2V^*)$ we find that
\begin{align*}
\scal(R) &= \tr(\Ric(R)) = \sum_{i=1}^n \langle \Ric(R)(b_i), b_i \rangle = \frac{1}{2} \sum_{i=1}^n \tr(R(b_i \wedge \cdot, b_i \wedge \cdot)) \\
&= \frac{1}{2} \sum_{i,j=1}^n R(b_i \wedge b_j, b_i \wedge b_j) =  \sum_{i<j} R(b_i \wedge b_j, b_i \wedge b_j) \\
&=  \tr(R).
\end{align*}
\end{bem}


\subsection{The evolution equation of the curvature operator} \label{secgeomstructures}

Given a solution to the Ricci flow on a manifold $M$, we want to investigate how the geometry changes in time. For this, a connection on the tangent bundle over the manifold $M \times \R$ plays a central role. In this section, we are particularly interested in the evolution of the orthonormal frame bundle on an initial Riemannian manifold. Moreover, we formulate the evolution equation of the Riemannian curvature operator under the Ricci flow.  
\medskip

Given a  manifold $M$ equipped with a  Riemannian metric $g_0$, one can consider solutions $g_t$, $t \in [0,T)$, of the partial differential equation
\begin{align*} 
\frac{\partial}{\partial t}g_t = -2\ric_{g_t}
\end{align*}
starting at $g_0$, so-called \textit{solutions to the Ricci flow}. 
One is interested in how the Riemannian curvature operator $Rm_{g_t}$ changes in time. (Notice that here the Riemannian curvature operator $Rm_{g}$ of a Riemannian manifold $(M,g)$ is a section of the \textit{bundle of algebraic curvature tensors} $S^2_B(\Lambda^2T^*\!M)$, defined in such a way that the Riemannian curvature operator of the standard sphere is twice the identity.) To this end, we define a connection on space-time associated to such a solution $g_t$.

\begin{deff} \label{defnabla}
On the vector bundle $TM \rightarrow M \times \R$, we introduce a linear connection $\nabla$. On the times slices $M \times \{t\}$, it is given by the Levi-Civita connection $\nabla^{g_t}$ of the metric $g_t$, but it is \textit{not} the product connection: Let $\frac{\partial}{\partial t}$ be the vector field on $M \times \R$ given by $\frac{\partial}{\partial t}|_{(x,t)} := \dot{c}(t)$ for $(x,t) \in M \times \R$, where $c(t) := (x,t)$. Let further $X, Y \in \Gamma(M \times \R, TM)$ be time-dependent vector fields on $M$. Then $\nabla$ is given by
\begin{align*} \begin{split} 
(\nabla_YX)|_{(x,t)} &:= \big(\nabla^{g_t}_{Y(\cdot,t)}X(\cdot, t)\big)|_x \\
\big(\nabla_{\!\frac{\partial}{\partial t}}X\big)\big{|}_{(x,t)} &:=  \frac{\partial}{\partial t}X(x,t) - \Ric_{g_t}(X(x,t))  
\end{split}
\end{align*}
for $(x,t) \in M \times \R$.
\end{deff}

The introduced connection above extends in the usual way to a connection on the vector bundle $S^2_B(\Lambda^2T^*\!M) \rightarrow M \times \R$ by the product rule.

\begin{bem}
$\nabla$ is metric in the sense that 
\begin{align*}
\frac{\partial}{\partial t}\big(g_t(X,Y)\big) = g_t \!\left(\nabla_{\!\frac{\partial}{\partial t}}X,Y\right) + g_t\!\left(X, \nabla_{\!\frac{\partial}{\partial t}}Y\right)
\end{align*}
for $X,Y \in \Gamma(M \times \R, TM)$.
\end{bem}

Recall that the \textit{orthonormal frame bundle} $O^{g_t}$ of $(M,g_t)$ is the principle bundle with structure group $O(n)$, the fibres over points $x \in M$ of which are given by
\begin{align*}
O^{g_t}_x := \{ p: (\R^n, \langle \cdot, \cdot \rangle) \rightarrow (T_xM, (g_t)_x) \mid p \ \text{is an linear isometry}\},
\end{align*}
where $\langle \cdot, \cdot \rangle$ is the standard metric on $\R^n$. Hence, the group $O(n)$ acts freely and transitively on the fibres of $O^{g_t}$ from the right. 
\medskip

\begin{deff} \label{ptparallel}
Let $x_0 \in M$ and $p_t: \R^n \rightarrow T_{x_0}M$ be a one-parameter family of linear maps. Then $t \mapsto p_t$ is \textit{parallel}, if $t \mapsto p_t(v)$ is parallel along the curve $t \mapsto (x_0,t)$ with respect to the connection $\nabla$ (introduced in Definition \ref{defnabla}) for all $v \in \R^n$, i.e.\ if 
\begin{align*}
\frac{d}{d t}p_t(v) = \Ric_{g_t}(p_t(v))
\end{align*}
for all $v \in \R^n$.
\end{deff}

Note that Definition \ref{ptparallel} is made in such a way that if $t \mapsto p_t$ is parallel and $t \mapsto R_t \in S^2_B(\Lambda^2T^*_{x_0}M)$ is smooth, then
\begin{align*}
\frac{d}{dt}(p_t^*R_t) = p_t^*\nabla_{\!\frac{\partial}{\partial t}}R_t.
\end{align*}

Moreover, we have that if $t \mapsto p_t$ is parallel and $p_0$ is an isometry with respect to the standard metric $\langle \cdot, \cdot \rangle$ on $\R^n$ and $g_0$ on $T_{x_0}M$, then $p_t$ is an isometry with respect to $\langle \cdot, \cdot \rangle$ and $g_t$ for all $t \in \R$ as well. According to the Picard-Lindelöf theorem, this shows that given an initial isometry $p_0$ there is always a parallel curve $t \mapsto p_t \in O^{g_t}$ starting at $p_0$. Hence, the Ricci flow preserves the bundle $O^{g_.}$. 
\medskip

Using the connection $\nabla$ on the vector bundle $TM \rightarrow M \times \R$ defined in Definition \ref{defnabla},  the evolution equation of the Riemannian curvature operator under the Ricci flow is  given as follows.

\begin{lem}{\normalfont(\cite[p.155]{4-mfdsHamilton})} \label{evolutionRmallg} 
If $g_t$ is a solution to the Ricci flow on a  manifold $M$, then the Riemannian curvature operator of $g_t$ evolves under the partial differential equation
\begin{align*} 
\nabla_{\!\frac{\partial}{\partial t}}{Rm}_{g_t} = \Delta_{g_t}{Rm}_{g_t} + {Rm}_{g_t}^2 + {Rm}_{g_t}^{\#}.
\end{align*}
\end{lem}

Here, $\Delta_g=(\nabla^g)^*\nabla^g$ is the Laplace operator acting on sections $R$ of the bundle of algebraic curvature tensors $S^2_B(\Lambda^2T^*\!M)$. One way to define this is by the formula 
\begin{align*} 
(\Delta_g R)(x) := \sum_{i=1}^n \frac{(\nabla^g)^2}{ds^2}\bigg{|}_{s=0} R(\gamma_i(s)),
\end{align*}
where $x \in M$, $\gamma_1, \dots, \gamma_n$ are $g$-geodesics in $M$ such that $\gamma_i(0) = x$ for $i=1, \dots, n$ and $(\dot{\gamma_1}(0), \dots, \dot{\gamma_n}(0))$ is an orthonormal basis of $T_xM$ with respect to $g$.


\section{Curvature conditions and ODE-invariance}

This chapter is dedicated to introducing curvature conditions, that is $O(n)$-invariant subsets $\Omega$ of the space of algebraic curvature tensors $\A$. To these sets, one can associate subsets $\Omega^g$ of the bundle of algebraic curvature tensors over a Riemannian manifold $(M,g)$, which are invariant under parallel transport by the Levi-Civita connection $\nabla^g$ of $(M,g)$. Using this notation, we are able to say when a Riemannian metric satisfies a given curvature condition.

Furthermore, we consider subsets of a vector space which are invariant under an ordinary differential equation of the form $f'(t)=\Phi(f(t))$ and give a characterization of these in terms of their tangent cones. 

\subsection{Curvature conditions} \label{secCurvCond}

Let $(M,g)$ be an $n$-dimensional  Riemannian manifold and $O^g$ the orthonormal frame bundle on $(M,g)$. There is a left-action of $O(n)$ on the space of algebraic curvature tensors, namely the representation 
\begin{align*}
\rho: O(n) \rightarrow \mathrm{End}(\A)
\end{align*}
of $O(n)$ on $\A = S^2_B(\Lambda^2(\R^n)^*)$ given by
\begin{align*}
(\rho(Q)R)(v \wedge w, y\wedge z) := R(Q^{-1}v \wedge Q^{-1}w, Q^{-1}y \wedge Q^{-1}z),
\end{align*}
where $Q \in O(n)$, $R \in \A$ and $v, w, y, z \in \R^n$. 
\medskip

A subset $\Omega \subseteq \A$ is called \textit{$O(n)$-invariant}, if for all $R \in \Omega$ we have that
\begin{align*}
\rho(Q)R \in \Omega
\end{align*}
for every $Q \in O(n)$.
To an $O(n)$-invariant subset $\Omega \subseteq \A$, we associate the subset $\Omega^g \subseteq S^2_B(\Lambda^2T^*\!M)$ defined by
\begin{align*}
\Omega^g := \left\{ R \in S^2_B(\Lambda^2T^*\!M) \mid p^*R \in \Omega \ \text{for some} \ p \in O^g_{\pi(R)} \right\},
\end{align*}
where $\pi: S^2_B(\Lambda^2T^*\!M) \rightarrow M$ is the projection map and $p^*R \in \A$ is the pullback of $R$ along $p$, i.e.\
\begin{align*}
(p^*R)(v \wedge w, x \wedge y) := R(p(v) \wedge p(w), p(x) \wedge p(y))
\end{align*}
for $v,w,x,y \in \R^n$. This is well-defined due to the $O(n)$-invariance of $\Omega$.
\medskip

Note that $\Omega^g$ is invariant under parallel transport by $\nabla^g$. Furthermore, if $\Omega$ is open respectively closed, $\Omega^g$ is open respectively closed as well.

\begin{deff}
Let $\Omega \subseteq \A$ be $O(n)$-invariant. We say that \textit{$g$ satisfies $\Omega$}, if for all $x \in M$, we have that
\begin{align*}
Rm_g(x) \in \Omega^g.
\end{align*}
Therefore, we often call such a set $\Omega$ \textit{a curvature condition}.
\end{deff}

\begin{deff} \label{invariantUnderRF}
We say that an $O(n)$-invariant set $\Omega \subseteq \A$ is \textit{invariant under the Ricci flow}, if for all $n$-dimensional compact manifolds $M$ and solutions $g_t$, $t \in [0,T)$, to the Ricci flow on $M$ with $g_0$ satisfying $\Omega$, we have that $g_t$ satisfies $\Omega$ for all $t \in [0,T)$.
\end{deff}


\subsection{Properties of tangent cones}

For closed subsets $C$ of a metric space with smooth boundary $\partial C$, we can linearly approximate the submanifold $\partial C$ at some point $x_0 \in \partial C$, namely by the tangent space $T_{x_0}\partial C$. If the boundary of $C$ is not smooth this concept fails. However, tangent cones of such subsets generalize this notion to arbitrary regularity of the boundary. In this section, we show some properties of tangent cones, which, in the non-smooth case, involves approximating the boundaries of the subsets by certain submanifolds, so-called supporting submanifolds.
\medskip

\begin{deff} \label{DefTangentcone}
Let $V$ be a metric space, $C \subseteq V$ be a closed subset and $x_0 \in C$. The \textit{tangent cone of $C$ at $x_0$} is defined by
\begin{align*}
T_{x_0}C := \overline{\{ \dot{\gamma}(0) \mid \gamma: (-\epsilon, \epsilon) \rightarrow V \ \text{in} \ \mathcal{C}^1 \ \text{with} \ \gamma(0)=x_0 \ \text{and} \ \gamma(t) \in C \ \text{for all} \ t \in [0, \epsilon)  \}}.
\end{align*}
\end{deff}

From now on, let $V$ be a Euclidean vector space with induced norm $\|\cdot\|$ and let $C \subseteq V$ be a closed subset.

\begin{lem}
Let $x_0 \in C$. Then we have that
\begin{align*}
T_{x_0}C \subseteq K_{x_0}C := \{ v \in V \mid \forall x\in V \ \text{with} \ d(x,C) = \|x_0-x\|: \langle x-x_0,v \rangle \leq 0 \}.
\end{align*}
\end{lem}

\begin{proof}
For $x_0$ being in the interior of $C$, we have that $K_{x_0}C = V$. Hence, in this case the statement is trivial. Now let $x_0 \in \partial C$ and $v \in T_{x_0}C$ such that $v=\dot{\gamma}(0)$, where $\gamma: (-\epsilon, \epsilon) \rightarrow V$ is a $\mathcal{C}^1$-curve with $\gamma(0)=x_0$ and $\gamma(t) \in C$ for all $t \in [0, \epsilon)$. Let $x \in V$ with $d(x,C) = \|x_0-x\|$. If we had that $\langle x-x_0,v \rangle >0$, then for $t>0$ small enough, we would find that
\begin{align*}
\|x-\gamma(t)\|^2 &= \| x - \gamma(0) - t\dot{\gamma}(0) + o(t) \|^2 = \|x-x_0\|^2 - 2t \underbrace{\langle x-x_0,v \rangle}_{>0} + o(t) < \|x-x_0\|^2 ,
\end{align*}
in contradiction to $\|x-x_0\| = d(x,C) \leq \|x-\gamma(t)\|$ since $\gamma(t) \in C$. Thus, $\langle x-x_0,v \rangle \leq 0$ and therefore $v \in K_{x_0}C$. The statement follows since $K_{x_0}C$ is closed.
\end{proof}

If the boundary $\partial C$ of $C$ is smooth and of codimension one, then for $x_0 \in \partial C$, we have that
\begin{align*}
T_{x_0}C = \{v \in V \mid \langle v, \mathbf{n}_{x_0} \rangle \leq 0 \},
\end{align*}
where $\mathbf{n}_{x_0}$ denotes the outward pointing unit normal on $\partial C$ at $x_0$.
If the boundary of $C$ is of lower regularity, approximating it pointwise by certain  submanifolds of $V$, a similar, however slightly weaker, result is true (see Lemma \ref{tangentconeinhalfspacessm}). To this end, we introduce the notion of a supporting submanifold.

\begin{deff}
Let $x_0 \in \partial C$. A \textit{supporting submanifold of $C$ in $x_0$} is a  submanifold $N$ of $V$ of codimension one that touches $C$ in $x_0$ such that $C$ locally lies on one side of $N$, meaning that there is an open neighborhood $U \subseteq V$ of $x_0$ such that $U \setminus N$ consists of exactly two connected components $U_1$ and $U_2$, the closure of one of those (say $U_1$) containing $C \cap \overline{U}$. \\
Moreover, by $r^N$, we will always denote a \textit{signed distance function from a supporting submanifold $N$ of $C$ in $x_0$}, i.e.\ a function
\begin{align*}
r^N: U \rightarrow \R: x \mapsto
&\left.\begin{cases}
- d(x,N), & x \in U_1 \\
d(x,N), & x \in U_2 
\end{cases} \right\} = 
\begin{cases}
- d(x,N), & x \ \text{lies on the side of} \ C \\
d(x,N), & \text{else}.
\end{cases} 
\end{align*}
By possibly making $U$ smaller, we can always arrange $r^N$ to be smooth, which we will assume throughout.
\end{deff}

\begin{lem} \label{tangentconeinhalfspacessm}
Let $x_0 \in \partial C$ and $N$ be a supporting submanifold of $C$ in $x_0$. Then we have that
\begin{align*}
T_{x_0}C \subseteq \{ v \in V \mid \langle v, \mathbf{n}_{x_0}  \rangle \leq 0 \} =: H_N,
\end{align*}
where $\mathbf{n}_{x_0}$ denotes the unit normal on $N$ at $x_0$ pointing in the opposite direction of $C$. In particular, the tangent cone $T_{x_0}C$ lies on one side of the tangent space $T_{x_0}N$.
\end{lem}

\begin{proof}
Let $\gamma: (-\epsilon, \epsilon) \rightarrow V$ be once differentiable with $\gamma(0)=x_0$ and $\gamma(t) \in C$ for all $t \in [0,\epsilon)$. Then $\dot{\gamma}(0) \in T_{x_0}C \subseteq V$ and we have that $r^N(\gamma(t)) \leq 0$ for all $t \in [0,\epsilon)$ and $r^N(\gamma(0)) = 0$. Hence,
\begin{align*}
0 \geq \frac{d}{dt}\bigg{|}_{t=0} r^N(\gamma(t)) = dr^N_{\gamma(0)}(\dot{\gamma}(0)) = \langle \mathbf{n}_R, \dot{\gamma}(0) \rangle.
\end{align*}
Passing to the closure yields that $T_{x_0}C \subseteq H_N$. Since the tangent space $T_{x_0}N$ is the boundary of the half space $H_N$, the tangent cone $T_{x_0}C$ lies on one side of $T_{x_0}N$. 
\end{proof}


\subsection{Invariance under an ordinary differential equation}

Let $V$ be a vector space and $\Phi: V \rightarrow V$ a locally Lipschitz continuous map. A subset $C \subseteq V$ is \textit{invariant under the ordinary differential equation} 
\begin{align} \label{allgODE}
f'(t) = \Phi(f(t)),
\end{align}
if for all solutions $f: [0,\delta] \rightarrow V$ of \eqref{allgODE} with $f(0) \in C$, we have that $f(t) \in C$ for all $t \in [0,\delta]$.
\medskip

In our applications, we will always have that $V= \A$ and that $\Phi: \A \rightarrow \A$ is the quadratic $O(n)$-equivariant map given by 
\begin{align*}
\Phi(R) := R^2 + R^{\#}
\end{align*}
for all $R \in \A$. 
\medskip

The following proposition is somewhat more general than Hamilton's statement in \cite[Lemma 4.1]{4-mfdsHamilton} since the convexity assumption is not required. 

\begin{prop} \label{ODETangkegel}
Let $V$ be a normed vector space, $\Phi: V \rightarrow V$ locally Lipschitz continuous and $C \subseteq V$ a closed set. Then $C$ is invariant under the ordinary differential equation \eqref{allgODE} if and only if for all $v \in \partial C$ we have that $\Phi(v) \in T_vC$.
\end{prop}

\begin{proof}
First, assume that $C$ is invariant under \eqref{allgODE}. Let $v \in \partial C$ and $f: [0,\delta] \rightarrow V$ be a solution of \eqref{allgODE} with $f(0)=v$. By the theorem of Picard-Lindelöf, $f$ is defined on the interval $(-\epsilon,\epsilon)$ for an $\epsilon \in(0,\delta)$ as well. Moreover, $f'(0) = \Phi(v)$ and the invariance of $C$ under \eqref{allgODE} yields that $f(t) \in C$ for all $t \in [0,\delta]$. Hence, $\Phi(v) \in T_vC$.
\medskip

In order to show the opposite direction, let $f: [0,\delta] \rightarrow V$ be a solution of \eqref{allgODE} with $f(0)\in C$. Let $r:V\rightarrow [0,\infty)$ be the distance function from $C$, i.e.\ for $v \in V$ let
\begin{align*}
r(v) := d(v,C) = \inf_{w\in C}\|v-w\|.
\end{align*}
Moreover, for $t \in [0,\delta]$ we set
\begin{align*}
s(t) := r(f(t))^2.
\end{align*}
In general, the function $s: [0,\delta] \rightarrow [0,\infty)$ is not differentiable. Still we can define
\begin{align*}
s'(t) := \limsup_{h \searrow 0} \frac{s(t+h) - s(t)}{h} < \infty
\end{align*}
for $t \in [0,\delta)$, since $r$ is Lipschitz continuous and $f$ is once continuously differentiable. Let $r_0$ be the maximum of $s$ on $[0,\delta]$. Then
\begin{align*}
K:= \bigcup_{t \in [0,\delta]}B_{\sqrt{r_0}}(f(t))
\end{align*}
is compact. Since $\Phi$ is locally Lipschitz continuous, there exists a constant $L>0$ such that $\Phi |_K$ is $\frac{L}{2}$-Lipschitz continuous. Our goal is to show that $s'(t) \leq L s(t)$ for all $t \in [0,\delta)$. Because then for 
\begin{align*} 
g: [0,\delta] \rightarrow [0,\infty): t \mapsto s(t)e^{-Lt},
\end{align*}
we find that $g'(t) \leq e^{-Lt}(s'(t)-Ls(t)) \leq 0$ for all $t \in [0,\delta)$ and $g(0)=0$. Therefore, $g(t) \leq 0$ for all $t \in [0,\delta]$, hence $s(t) \leq 0$ for all $t \in [0,\delta]$. Since $s$ is non-negative, however, this means that $s \equiv 0$, which yields that $f(t) \in C$ for all $t \in [0,\delta]$.\\
Let now $t \in [0,\delta)$. Since $C$ is closed, there is an $x_t \in C$ with $d(f(t), C) = \|f(t) - x_t\|$. By assumption, $\Phi(x_t) \in T_{x_t}C \subseteq K_{x_t}C$, thus $\langle f(t)-x_t, \Phi(x_t) \rangle \leq 0$. Consequently,
\begin{align*}
s'(t) &= \limsup_{h \searrow 0} \frac{d(f(t+h),C)^2 - d(f(t),C)^2}{h} \leq  \limsup_{h \searrow 0} \frac{\|f(t+h)-x_t\|^2-\|f(t)-x_t\|^2}{h} \\
&= \limsup_{h \searrow 0} \frac{\|f(t+h)\|^2 - \|f(t)\|^2 - 2 \langle f(t+h) -f(t), x_t\rangle}{h} \\
&= \frac{d}{dt}\|f(t)\|^2 -2 \langle f'(t), x_t \rangle \stackrel{\eqref{allgODE}}{=} 2 \langle f'(t), f(t) \rangle - 2 \langle \Phi(f(t)), x_t \rangle \stackrel{\eqref{allgODE}}{=} 2 \langle \Phi(f(t)), f(t) - x_t \rangle \\
&\leq 2 \langle \Phi(f(t)), f(t) - x_t \rangle - 2 \langle  \Phi(x_t), f(t)-x_t \rangle = 2 \langle \Phi(f(t)) - \Phi(x_t), f(t)-x_t \rangle \\
&\leq 2 \|\Phi(f(t)) - \Phi(x_t)\| \|f(t)-x_t\| \leq L \|f(t)-x_t\|^2 = Ls(t),
\end{align*}
where the last inequality holds, since
\begin{align*}
\|x_t-f(t)\|^2 = s(t) \leq r_0,
\end{align*} 
thus $x_t \in B_{\sqrt{r_0}}(f(t)) \subseteq K$.
\end{proof}


\section{Bianchi-convex sets}

In this chapter, we introduce Bianchi-convex sets of algebraic curvature tensors and show some first properties. Bianchi-convexity relaxes the notion of convexity in a certain sense inspired by the second Bianchi identity for the Riemannian curvature tensor of a Riemannian manifold. In dimension three, we consider Bianchi-convex sets of algebraic curvature tensors whose eigenvalues lie in a sublevel set of some function $f: \R^3 \rightarrow \R$ and give another characterization of Bianchi-convexity for those sets in terms of $f$. This enables us to find examples for Bianchi-convex sets which are not convex, whichs shows that the introduced notion is a real generalization of convexity. Furthermore, certain subsets of these Bianchi-convex sets are invariant under the ordinary differential equation \eqref{ODEeinleitung}. 

\subsection{The definition and first properties}

First of all, recall that for a submanifold $N$ of codimension one of a Riemannian manifold $(M,g)$ and a point $x \in M$, given the choice of a normal vector $\mathbf{n}_x$ at $x$, the \textit{second fundamental form of $N$ in $x$} is defined as the symmetric and bilinear map
\begin{align*}
\text{\Gemini}^N_x: T_xN \times T_xN \rightarrow \R: (X,Y) \mapsto g_x( \nabla^g_XY, \mathbf{n}_x ),
\end{align*}
where $\nabla^g$ is the Levi-Civita connection of $M$. For $X, Y \in T_xN$, one can show that 
\begin{align} \label{nablaMitMinusRüberholen}
g_x( \nabla^g_XY, \mathbf{n}_x ) = - g_x(Y,\nabla^g_X\mathbf{n}).
\end{align}
Here, $\mathbf{n}$ denotes an extension of $\mathbf{n}_x$ to a neighborhood of $x$ in $M$.

\begin{bem} \label{convexnegsemidef}
Let $V$ be a vector space and $C \subseteq V$ a closed convex set, the boundary $\partial C$ of which is smooth and of codimension one. If one chooses $\mathbf{n}_x$ to be the \textit{outward} pointing unit normal on $\partial C$ at $x$, then ${\text{\Gemini}}^{\partial C}_x$ is negative semidefinite for all $x \in \partial C$.
\end{bem}

Our generalization of the notion of convexity requires a second Bianchi identity for tuples of algebraic curvature tensors.

\begin{deff}
An $n$-tuple $(T_1, \dots, T_n) \in \A^n$ \textit{satisfies the second Bianchi identity}, if for some orthonormal basis $(b_1, \dots, b_n)$ of $\R^n$, we have that
\begin{align*}
T_i(b_j \wedge b_k) + T_j(b_k \wedge b_i) + T_k(b_i \wedge b_j) = 0
\end{align*}
for all $i,j,k \in \{1, \dots, n\}$. Moreover, we can replace $\A$ by $S^2_B(\Lambda^2T^*_xM)$ and $\R^n$ by $T_xM$ for some $x \in M$.
\end{deff}

\begin{bsp}
Let $(M,g)$ be an $n$-dimensional Riemannian manifold, $x \in M$ and $(b_1, \dots, b_n)$ an orthonormal basis of $T_xM$. Then $(T_1, \dots, T_n)$, where $T_i := \nabla_{b_i} Rm_g$ for $i=1, \dots, n$, satisfies the second Bianchi identity with respect to $(b_1, \dots, b_n)$.
\end{bsp}

For closed subsets of algebraic curvature tensors, we introduce a weaker form of convexity.

\begin{deff} \label{DefBianchiconvexSet}
A closed subset $\Omega \subseteq \A$ is called \textit{Bianchi-convex}, if for all $\epsilon > 0$ and $R \in \partial\Omega$ there is a supporting submanifold $N$ of $\Omega$ in $R$ such that for each $S \in N$ and $(T_1, \dots, T_n) \in (T_SN)^n$ satisfying the second Bianchi identity, we have that
\begin{align} \label{glgallgBCsets}
\sum_{i=1}^n \text{\Gemini}^{N}_S (T_i,T_i) \leq \epsilon \sum_{i=1}^n \|T_i\|^2.
\end{align}
Furthermore, we can replace $\A$ by $S^2_B(\Lambda^2T^*_xM)$ for some $x \in M$.
\end{deff}

If the boundary of $\Omega$ is smooth, then the supporting submanifolds in the Definition \ref{DefBianchiconvexSet} can be chosen to be $\partial\Omega$ itself, and we obtain that $\Omega$ is Bianchi-convex, if and only if for all $R \in \partial \Omega$ and $(T_1, \dots, T_n) \in (T_R\partial \Omega)^n$ satisfying the second Bianchi identity, we have that
\begin{align*}
\sum_{i=1}^n \text{\Gemini}_R^{\partial \Omega}(T_i,T_i) \leq 0.
\end{align*}

Roughly speaking, in order for a set of algebraic curvature tensors to be Bianchi-convex, concavity is permitted in certain directions as long as these directions are compensated by the convex ones.

In particular, closed convex subsets of $\A$ are Bianchi-convex (c.f.\ Remark \ref{convexnegsemidef}).

\begin{lem} \label{SchnitteWiederBkonvex}
The intersection of two Bianchi-convex sets is Bianchi-convex.
\end{lem}

\begin{proof}
Let $\Omega, \Omega' \subseteq \A$ be Bianchi-convex sets. Let $\epsilon > 0$ and $R \in \partial(\Omega \cap \Omega')$. 
If $R \in \partial \Omega \cap \Omega' \subseteq \partial\Omega$, then, since $\Omega$ is Bianchi-convex, there exists a supporting submanifold $N$ of $\Omega$ in $R$ such that \eqref{glgallgBCsets} is true for each $S \in N$ and $(T_1, \dots, T_n) \in (T_SN)^n$ that satisfies the second Bianchi identity. Since $N$ is also a supporting submanifold of $\Omega \cap \Omega'$ in $R$, we are done in this case.
The case that $R \in \partial \Omega' \cap \Omega \subseteq \partial\Omega'$ works analogously. \qedhere
\end{proof}

The following lemma is a crucial step towards generalizing Hamilton's maximum principle \cite[Theorem 4.3]{4-mfdsHamilton} to the Bianchi-convex setting.

\begin{lem} \label{laplacezeigtnachinnen}
Let $(M,g)$ be an $n$-dimensional Riemannian manifold, $C \subseteq S^2_B(\Lambda^2T^*\!M)$ be closed, invariant under parallel transport with respect to the Levi-Civita connection $\nabla^g$ and fibrewise Bianchi-convex. Let $R \in \Gamma(M,C)$ be a smooth section of $C$ and $R(x) \in \partial C_x$ for some point $x \in M$. Moreover, assume that $(\nabla^g_{b_1}R, \dots, \nabla^g_{b_n}R)$ satisfies the second Bianchi identity with respect to some orthonormal basis $(b_1, \dots, b_n)$ of $T_xM$. Let further $\epsilon > 0$ and $N$ be a supporting submanifold of $C_x$ in $R(x)$ satisfying \eqref{glgallgBCsets}. Then we have that
\begin{align*}
\left\langle (\Delta_g R)(x), \mathbf{n}_{R(x)} \right\rangle_g \leq \epsilon \sum_{i=1}^n \| \nabla^g_{b_i}R \|_g^2,
\end{align*}
where $\mathbf{n}_{R(x)}$ denotes the  unit normal on $N$ at $R(x)$ pointing in the opposite direction of $C_x$.
\end{lem}

\begin{proof}
For $i=1, \dots, n$, let  $\gamma_i: (-\delta, \delta) \rightarrow M$ be geodesics with $\gamma_i(0)=x$ and $\dot{\gamma_i}(0) =b_i$. Since $C$ is invariant under parallel transport with respect to $\nabla^g$, for $i=1,\dots,n$ and $t \in (-\delta, \delta)$ we find that
\begin{align*}
h_i(t) := \left(P_{\gamma_i|_{[0,t]}}\right)^{-1} \underbrace{(R \circ \gamma_i)(t)}_{\in C_{\gamma_i(t)}} \in C_x, 
\end{align*} 
where $P_{\gamma_i|_{[0,t]}}$ denotes the parallel transport along $\gamma_i|_{[0,t]}$ with respect to $\nabla^g$. Hence, $h_i$ is a curve in $C_x$ with $h_i(0)=R(x) \in N$. Now, let $r^N$ be a signed distance function from $N$. Then 
\begin{align*}
(r^N \circ h_i)(0) &= r^N(R(x)) = 0, \\
\text{and} \ \ \ \ (r^N \circ h_i)(t) &\leq 0 \ \ \ \ \text{for all} \ t \in (-\delta,\delta).
\end{align*}
Therefore, $0$ is a local maximum of $r^N \circ h_i$. On the one hand, this implies that
\begin{align*}
0 = \frac{d}{dt}\bigg{|}_{t=0} (r^N \circ h_i)(t) = dr^N_{h_i(0)}(h_i'(0)) = \langle \grad_{h_i(0)}r^N, h_i'(0) \rangle = \left\langle \mathbf{n}_{h_i(0)}, h_i'(0) \right\rangle,
\end{align*}
thus 
\begin{align*}
\nabla^g_{b_i}R =\frac{\nabla^g}{dt}\bigg{|}_{t=0}R(\gamma_i(t)) =h_i'(0) \in T_{h_i(0)}N. 
\end{align*}
On the other hand, we obtain that
\begin{align*}
0 \geq \frac{d^2}{dt^2}\bigg{|}_{t=0} (r^N\circ h_i)(t) = \Hess_{h_i(0)}r^N\big(h_i'(0), h_i'(0)\big) + dr^N_{h_i(0)}(h_i''(0)).
\end{align*}
Since
\begin{align*}
\Hess_{h_i(0)}r^N\big(h_i'(0), h_i'(0)\big) &= \left\langle \nabla^g_{h_i'(0)}\grad \, r^N, h_i'(0) \right\rangle \stackrel{\eqref{nablaMitMinusRüberholen}}{=} - \left\langle \grad_{R(x)} r^N, \nabla^g_{h_i'(0)}h_i'(0) \right\rangle \\
&= - \left\langle \mathbf{n}_{R(x)},\nabla^g_{h_i'(0)}h_i'(0)  \right\rangle = - \text{\Gemini}^N_{R(x)}(h_i'(0), h_i'(0)),
\end{align*}
we find that
\begin{align} \label{unglg22fsp}
\text{\Gemini}^N_{R(x)}(h_i'(0), h_i'(0)) \geq \left\langle \mathbf{n}_{R(x)}, h_i''(0) \right\rangle,
\end{align}
which leads to
\begin{align*}
\left\langle (\Delta_g R)(x), \mathbf{n}_{R(x)} \right\rangle &= \sum_{i=1}^n \left\langle \frac{(\nabla^g)^2}{dt^2}\bigg{|}_{t=0}R(\gamma_i(t)), \mathbf{n}_{R(x)} \right\rangle = \sum_{i=1}^n \left\langle h_i''(0), \mathbf{n}_{R(x)} \right\rangle \\ 
&\stackrel{\eqref{unglg22fsp}}{\leq} \sum_{i=1}^n \text{\Gemini}^N_{R(x)}(h_i'(0), h_i'(0)) = \sum_{i=1}^n \text{\Gemini}^N_{R(x)}(\nabla^g_{b_i}R, \nabla^g_{b_i}R) \\
&\stackrel{\eqref{glgallgBCsets}}{\leq} \epsilon \sum_{i=1}^n \|\nabla^g_{b_i}R\|^2,
\end{align*}
where the last inequality holds since $(\nabla^g_{b_1}R, \dots, \nabla^g_{b_n}R) \in (T_{R(x)}N)^n$ satisfies the second Bianchi identity (with respect to the orthonormal basis $(b_1, \dots, b_n)$) and $C_x$ is Bianchi-convex.
\end{proof}


\subsection{Bianchi-convex sets in dimension three} \label{BconvexSetsInDim3}

In this section, we give an example of a Bianchi-convex set which is not convex. To this end, we give another characterization of Bianchi-convex sets which have the following concrete form.
\medskip

Let $f: \R^3 \rightarrow \R$ be a smooth function with $f^{-1}(0) \neq \emptyset$ such that 0 is a regular value of $f$. Let further $f$ be symmetric, meaning that $f(x_1,x_2,x_3)=f(x_{\sigma(1)},x_{\sigma(2)},x_{\sigma(3)})$ for all $x \in \R^3$ and all permutations $\sigma \in S_3$. We set
\begin{align*}
\Omega_f := \{ R \in \mathcal{A}_3 \mid f(\lambda(R)) \leq 0 \}.
\end{align*}
Here,
\begin{align*}
\lambda: \mathcal{A}_3 \rightarrow \R^3: R \mapsto (\lambda_1(R), \lambda_2(R), \lambda_3(R)),
\end{align*} 
where $\lambda_1(R) \leq \lambda_2(R) \leq \lambda_3(R)$ are the eigenvalues of $R$. Note that by the assumptions on $f$, the boundary of $\Omega_f$ is smooth.

\begin{bsp}
The following functions satisfy the properties mentioned above:
\begin{align*}
x &\mapsto x_1 + x_2 + x_3 \\
\text{or} \ \ x &\mapsto x_1^2 + x_2^2 + x_3^2 - a(x_1+x_2+x_3)^2 - c
\end{align*}
for $a \in \R$ and $c >0$.
\end{bsp}

The sublevel sets of a convex function are convex, as is well known. Since $\Omega_f$ is a sublevel set of $f \circ \lambda$, it is natural to ask whether $f$ being convex implies that $\Omega_f$ is already convex. Since by assumption $f$ is symmetric, the answer is yes: If $f$ is convex, then each connected component of $\Omega_f$ is convex. (For a proof see \cite[Lemma~3.32]{ThesisSFB}.) Requiring only Bianchi-convexity from $\Omega_f$ allows to relax the convexity condition; the following theorem was shown by the author in her PhD thesis \cite{ThesisSFB}, and it shows that $f$ is allowed to have a certain degree of concavity in some directions.

\begin{thm} \label{characterizationBconvexSets}
The set $\Omega_f$ is Bianchi-convex if and only if for all $\lambda \in f^{-1}(0)$ with $\lambda_1 < \lambda_2 < \lambda_3$ the following are true.
\begin{enumerate}
\item[1.)] $\partial_1f(\lambda) \leq \partial_2f(\lambda) \leq \partial_3f(\lambda)$ \text{and,}
\item[2.)] unless $Z_1=Z_2=Z_3=0$, we have for all $x \in \ker(df_{\lambda})$  that
\begin{align*}
\Hess_{\lambda}f(x,x) \geq -2 \min\limits_{\substack{\{i,j,k\}\\=\{1,2,3\}}} \frac{Z_iZ_j}{Z_i+Z_j}x_k^2,
\end{align*}
while if $Z_1=Z_2=Z_3=0$, we have that $\Hess_{\lambda}f(x,x) \geq 0$ for all $x \in \ker(df_{\lambda})$. Here, for $i=1,2,3$, we denote
\begin{align*}
Z_i := Z_i(\lambda) := \frac{\partial_kf(\lambda)- \partial_jf(\lambda)}{\lambda_k-\lambda_j}
\end{align*}
for $\{i,j,k\} = \{1,2,3\}$. 
\end{enumerate}
\end{thm}

\begin{bem} 
Using Theorem \ref{characterizationBconvexSets}, one can prove that if $\Omega_f$ is Bianchi-convex and scale-invariant, then the connected components of $\Omega_f$ are convex, c.f. \cite[Prop.~3.43]{ThesisSFB}.
\end{bem}

We are now in the position to show that Bianchi-convexity is a genuine generalization of convexity, i.e.\  there are Bianchi-convex sets which are \textit{not} convex.

\begin{prop} \label{BSPstriktBkonvexnichtkonvex}
For $a \in \big(\frac{1}{3}, \frac{2}{5}\big)$ and $c>0$, the set
\begin{align*}
\Omega_{a,c} := \{ R \in \mathcal{A}_3 \mid \|R\|^2 - a \, \scal(R)^2 \leq c \}
\end{align*}
is Bianchi-convex but not convex.
\end{prop}

\begin{proof}
For $a \in \big(\frac{1}{3}, \frac{2}{5}\big)$ and $c >0$, we consider the function
\begin{align*}
f_{a,c}: \R^3 \rightarrow \R: x &\mapsto \|x\|^2 - a \langle x, I \rangle^2 - c ,
\end{align*}
where $I:= (1,1,1)^t$.
Obviously, $f_{a,c}$ is smooth and symmetric. For $x := \big(\sqrt{\frac{c}{1-a}},0,0\big)^t\in \R^3$, we have that $f_{a,c}(x)=0$, thus $f_{a,c}^{-1}(0) \neq \emptyset$. Moreover, $0$ is a regular value of $f_{a,c}$, since $c \neq 0$. Remembering Remark \ref{scal=2tr}, we obtain that $\Omega_{a,c} = \Omega_{f_{a,c}}$. Applying  Proposition \ref{characterizationBconvexSets} and using that $a < \frac{2}{5}$ yields that $\Omega_{a,c}$ is  Bianchi-convex. Moreover, it is easy to verify that $\Omega_{a,c}$ is not convex.
\end{proof}

\begin{center}
\begin{minipage}[b]{0.6\textwidth}
\vspace{5mm}
\includegraphics[width=0.7\textwidth]{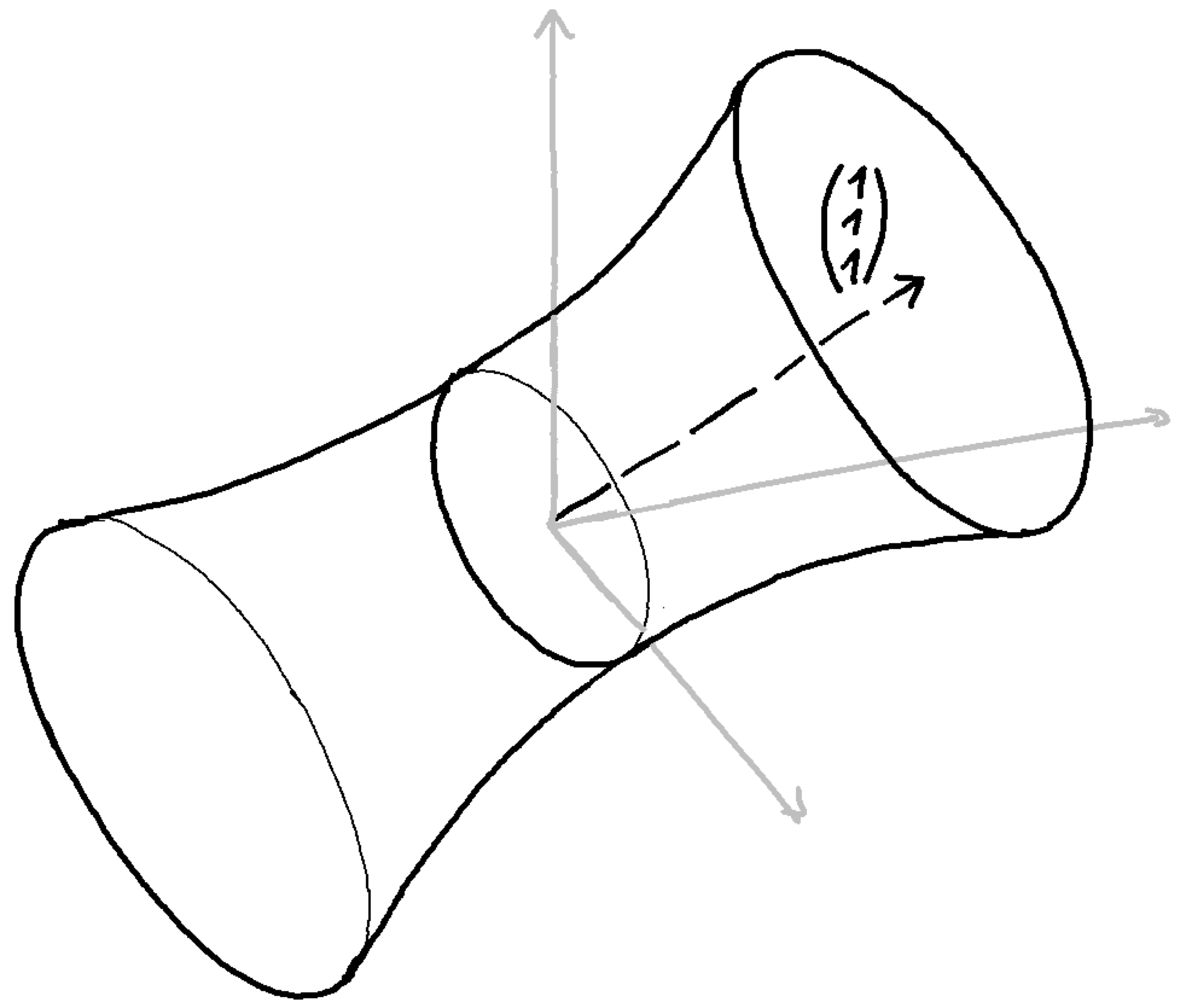}
\centering 
\end{minipage}
\begin{minipage}[t]{0.5\textwidth}
\vspace{-1mm}
\ \ \ \ Sketch of $f^{-1}_{a,c}(0)$ for $a \in \left(\frac{1}{3}, \frac{2}{5}\right)$ and $c>0$
\end{minipage}
\end{center}


\begin{bem}
From the proof of Proposition \ref{BSPstriktBkonvexnichtkonvex}, it even follows that for $a \in \big(\frac{1}{3}, \frac{2}{5}\big)$ and $c>0$, the set $\Omega_{a,c}$ is \textit{strictly} Bianchi-convex, that is for all $R \in \partial \Omega_{a,c}$ and $(T_1,T_2,T_3) \in (T_R\partial \Omega_{a,c})^3 \setminus \{0\}$ satisfying the second Bianchi identity, the inequality
\begin{align*}
\sum_{i=1}^3 \text{\Gemini}^{\partial \Omega_{a,c}}_R(T_i,T_i) <0
\end{align*}
is strict. Consequently, small deformations of the sets $\Omega_{a,c}$ are still Bianchi-convex.
\end{bem}

In the author's thesis \cite{ThesisSFB}, it was shown that suitable subsets of the sets $\Omega_{a,c}$ are invariant under the ordinary differential equation \eqref{ODEeinleitung}:

\begin{prop} \label{3DMengeODEinv}
For every $a \in \big(\frac{1}{3}, \frac{2}{5}\big)$ and $c>0$, there exists a constant $b_{a,c}>0$ such that the intersection  
\begin{align*}
\widetilde{\Omega}_{a,c} := \Omega_{a,c} \cap \{R \in \mathcal{A}_3 \mid \scal(R) \geq b_{a,c}\} 
\end{align*}
is invariant under the ordinary differential equation \eqref{ODEeinleitung}. 
\end{prop}

Notice that since $\{R \in \mathcal{A}_3 \mid \scal(R) \geq b_{a,c}\}$ as a half space is convex and hence also Bianchi-convex, with a view on Lemma \ref{SchnitteWiederBkonvex}, the set $\widetilde{\Omega}_{a,c}$ as intersection of Bianchi-convex sets is still Bianchi-convex. Moreover, a possible choice of the constant $b_{a,c}$ is 
\begin{align*} 
b_{a,c}:= \sqrt{\frac{3c}{3a-1}} \sinh\!\left(\frac{3}{2}\right).
\end{align*}

It is easy to verify that the sets $\widetilde{\Omega}_{a,c}$ as in Proposition \ref{3DMengeODEinv} are uniformly transversally star-shaped with respect to $\lambda I$ for all $\lambda > b_{a,c}$. Therefore, Theorem A yields that these sets are invariant under the Ricci flow.

\begin{bem}
In dimensions $n \geq 4$, the space $\mathcal{A}_n$ of algebraic curvature tensors has a non-vanishing Weyl part, therefore it seems not appropriate to consider sets that are given in terms of their eigenvalues. However, even for sets analogous to the sets $\Omega_f$ above, a characterization of Bianchi-convexity is currently unknown. 

The three-dimensional example above can be generalized in the following way. Let $\Omega \subseteq \mathcal{A}_n$ be a closed cone containing the identity $I$. Since $\Omega$ is a cone (hence $\text{\Gemini}^{\partial \Omega}_R(R, R) = 0$ for all $R \in \partial \Omega$), it can not be strictly convex. However, if the {\em base} of $\Omega$ is strictly convex, $\Omega$ can still be strictly Bianchi-convex: for example, this is so if the opening angle of the cone is not too large, in the sense that $\sphericalangle(R, I)$ is small enough for all $R \in \Omega$ (c.f.\ the discussion in Section~5.2 of \cite{ThesisSFB}). This means that small perturbations of the set $\Omega$ are still Bianchi-convex, even though they may fail to be convex. Moreover, again if the opening angle of the cone is not too large, $\Omega$ will be invariant under \eqref{ODEeinleitung}, as the $I$ axis is an attractor of \eqref{ODEeinleitung}. The task is now to quantify this, i.e.\ to understand how large the opening angle of $\Omega$ may be (where the answer may depend on the direction perpendicular to $\Omega$). This is a topic of ongoing research.
\end{bem}


\section{The maximum principle} \label{MaximumPrinciples}

In this chapter, we give a reformulation of Hamilton's maximum principle in the special case of algebraic curvature tensors and the Ricci flow. The aim is to generalize this version to Bianchi-convex sets. 
\medskip

In the setting of algebraic curvature tensors and the Ricci flow, Hamilton's  maximum principle \cite[Theorem 4.3]{4-mfdsHamilton} can be formulated as follows.

\begin{thm} \label{Hamilton2}
Let $g_t$ be a solution to the Ricci flow on an $n$-dimensional manifold $M$. Let further $\Omega \subseteq \A$ be an $O(n)$-invariant, closed and convex  set.
If $\Omega$ is invariant under the ordinary differential equation \eqref{ODEeinleitung},
then the family of sets $\Omega^{g_t} \subseteq S^2_B(\Lambda^2T^*\!M)$ is invariant under the partial differential equation
\begin{align} \label{PDEHamiltonumformuliert}
\nabla_{\!\frac{\partial}{\partial t}}R_t = \Delta_{g_t}R_t + R_t^2 + R_t^{\#},
\end{align}
i.e.\ for solutions $R$ to \eqref{PDEHamiltonumformuliert} with $R_0(x) \in \Omega^{g_0}_x$ for all $x \in M$, we have that $R_t(x) \in \Omega^{g_t}_x$ for all $x \in M$ and $t \in [0,T_0)$. Here, $T_0 \leq T$ denotes the maximal existence time of $R$.
\end{thm}

For details concerning the reformulation in this language see the author's thesis \cite[Section 4.2]{ThesisSFB}. As an immediate consequence of Theorem \ref{Hamilton2}, we find

\begin{kor} \label{CorHamilton}
Let $\Omega \subseteq \A$ be an $O(n)$-invariant, closed and convex set. If $\Omega$ is invariant under the ordinary differential equation \eqref{ODEeinleitung}, then $\Omega$ is invariant under the Ricci flow (in the sense of Definition \ref{invariantUnderRF}).
\end{kor}

Next, we want to give a generalization of Hamilton's maximum principle in the setting of Corollary \ref{CorHamilton} to Bianchi-convex sets. In order to formulate this new maximum principle, we provide the following definition.

\begin{deff} \label{coneCondition}
A closed set $\Omega \subseteq \A$ is called \textit{uniformly transversally star-shaped with respect to $S \in \A$}, if for each compact set $K \subseteq \A$ there exists a constant $r>0$ such that for each $R \in K \cap \partial \Omega$ there is an $\varepsilon_0>0$ such that
\begin{align*}
R + \varepsilon B_r(S-R) \subseteq \Omega
\end{align*}
for all $\varepsilon \in [0, \varepsilon_0)$.
\end{deff}

Notice that if $\Omega$ is uniformly transversally star-shaped with respect to $S$, then for all $R \in \Omega$ we have that $S-R$ is in the interior of the tangent cone $T_R\Omega$.

\begin{thm} \label{meinMP}  
Let $\Omega \subseteq \A$ be $O(n)$-invariant, closed, Bianchi-convex and uniformly transversally star-shaped with respect to $\lambda I$ for some $\lambda \in \R$. If $\Omega$ is invariant under the ordinary differential equation \eqref{ODEeinleitung},
then $\Omega$ is invariant under the Ricci flow (see Definition \ref{invariantUnderRF})
\end{thm}

In Theorem \ref{meinMP}, $I$ denotes the identity in $\A$. 

\begin{bem} \label{HamiltonSetsUTS}
Each $O(n)$-invariant, closed and convex set $\Omega \subseteq \A$ is transversally star-shaped with respect to $\lambda I$ for some $\lambda \in \R$; more precisely, there is a $\lambda \in \R$ such that for each $R \in \partial \Omega$ we have that $R + \alpha (\lambda I - R) \in \Omega$ for all $\alpha \in [0,1]$. Namely, let $R_0 \in \Omega$. It is clear that
\begin{align*}
S := \frac{1}{\vol(O(n))} \int_{O(n)} \rho(Q)R_0 dQ 
\end{align*}
is invariant under the representation $\rho$ of $O(n)$, hence so is the space $\R S \subseteq \A$ of all multiples of $S$. It is well-known, however, that the only one-dimensional irreducible subrepresentation of $\A$ is the space $\R I$ of multiples of the identity. Therefore, there is a $\lambda \in \R$ such that $S = \lambda I$. Since $\Omega$ is $O(n)$-invariant (thus all integrants are contained in $\Omega$), closed and convex, the weighted integral $S$ (and therefore $\lambda I$) is contained in $\Omega$. Again due to the convexity of $\Omega$, this shows that the connecting line $(1-\alpha)R + \alpha \lambda I = R + \alpha (\lambda I - R)$ is in $\Omega$ for $\alpha \in [0,1]$. \\
For essentially all known $O(n)$-invariant, closed and convex sets which are invariant under \eqref{ODEeinleitung}, it is furthermore easy to see that they are even uniformly transversally star-shaped with respect to $\lambda I$ for some $\lambda \in \R$.
\end{bem}

The rest of the chapter is dedicated to a proof of the maximum principle, Theorem \ref{meinMP}, for which we need the following two auxiliary lemmas.

\begin{lem} \label{concondFolgtInterior}
If a closed set $\Omega \subseteq \A$ is uniformly transversally star-shaped with respect to $S \in \A$, then for each $R \in \Omega$, we have that $R + \alpha (S-R)$ is in the interior of $\Omega$ for all $\alpha \in (0,1)$.
\end{lem}

\begin{proof}
From the assumption, it immediately follows that $S-R \in T_R\Omega$ for all $R \in \partial \Omega$. Therefore, by Proposition \ref{ODETangkegel}, $\Omega$ is invariant under the ordinary differential equation 
\begin{align} \label{ODES-R}
R'(t) = S-R(t). 
\end{align}
Let $R \in \Omega$. Then $R(t) := R + (1-e^{-t})(S-R)$ is a solution to \eqref{ODES-R} with $R(0) = R$. Therefore, $R(t) \in \Omega$ for all $t \in [0,\infty)$, which by the closedness of $\Omega$ means that $R + \alpha (S-R) \in \Omega$ for all $\alpha \in [0,1]$. 
\medskip

Now, let $R \in \partial \Omega$, $K \subseteq \A$ be a compact set containing $R$ and $r >0$ be as in Definition \ref{coneCondition}. By assumption, there is an $\varepsilon_0 > 0$ such that $R + \varepsilon B_r(S-R) \subseteq \Omega$ for all $\varepsilon \in [0, \varepsilon_0)$. It follows that $(1-\alpha) \cup_{\varepsilon \in [0,\varepsilon_0)}(R +  \varepsilon B_r(S-R)) + \alpha S \subseteq \Omega$ for all $\alpha \in [0,1]$. Therefore, $(1-\alpha)R+ \alpha S$ is contained in the interior of $\Omega$ for all $\alpha \in (0,1)$.
\medskip

If $R$ is in the interior of $\Omega$, then there is a neighborhood $U$ of $R$ which is contained in the interior of $\Omega$ as well. Therefore, $(1-\alpha)U + \alpha S \subseteq \Omega$ for all $\alpha \in [0,1]$, which shows that $(1-\alpha)R+ \alpha S$ is in the interior of $\Omega$ for all $\alpha \in [0,1)$.
\end{proof}

\begin{lem} \label{coneCondImplnI}
If a closed set $\Omega \subseteq \A$ is uniformly transversally star-shaped with respect to $S \in \A$, then for each compact set $K \subseteq \A$, we have that
\begin{align} \label{asupnI}
-a := \sup \langle n,S-R \rangle <0,
\end{align}
where the supremum is taken over all $R \in K \cap \partial \Omega$ and $n \in \A$ with $\langle n,v \rangle \leq 0$ for all $v \in T_R\Omega$ and $\|n\|=1$ (i.e.\ outward pointing generalized normal vectors $n$ on $\partial \Omega$ at $R$).
\end{lem}

\begin{proof}
Let $K \subseteq \A$ be compact, $r >0$ as in Definition \ref{coneCondition}, $R \in K \cap \partial \Omega$ and $n \in \A$ with and $\|n\|=1$ and $\langle n,v \rangle \leq 0$ for all $v \in T_R\Omega$. By assumption, there is an $\varepsilon_0 > 0$ such that $R + \varepsilon B_r(S-R) \subseteq \Omega$ for all $\varepsilon \in [0, \varepsilon_0)$. Hence, $B_r(S-R) \subseteq T_R\Omega$ and by scale-invariance of $T_R\Omega$, it follows that $\R_{>0}B_r(S-R) \subseteq T_R\Omega$. Let $\alpha \in (0,\pi)$ denote the opening angle of the cone $\R_{>0}B_r(S-R)$. Then
\begin{align*}
\arccos\left( \frac{\langle n,S-R \rangle}{\|S-R\|} \right) = \sphericalangle(n,S-R) > \frac{\alpha}{2} + \frac{\pi}{2}
\end{align*}
and since $\frac{\alpha}{2} + \frac{\pi}{2} \in (\frac{\pi}{2},\pi)$, we find that
\begin{align*}
 \frac{\langle n,S-R \rangle}{\|S-R\|} < \cos \left( \frac{\alpha}{2} + \frac{\pi}{2} \right).
\end{align*}
Hence,
\begin{align*}
\langle n,S-R \rangle < \|S-R\| \cos \left( \frac{\alpha}{2} + \frac{\pi}{2} \right) \leq \max_{\widetilde{R} \in K\cap \partial \Omega} \|S-\widetilde{R}\| \cos \left( \frac{\alpha}{2} + \frac{\pi}{2} \right) <0.
\end{align*}
Since $\alpha$ only depends on $K$, this finishes the proof.
\end{proof}

Now, we are in the position to prove the maximum principle in the Bianchi-convex setting.

\begin{proof}[Proof of Theorem \ref{meinMP}.] 
Let $M$ be an $n$-dimensional compact manifold and $g_t$, $t \in [0,T)$, be a solution to the Ricci flow with $g_0$ satisfying $\Omega$, i.e.\ with $Rm_{g_0}(x) \in \Omega^{g_0}_x$ for all $x \in M$. Let $T_1 \in (0,T)$. We will show that $Rm_{g_t} \in \Omega^{g_t}$ for all $t \in [0,T_1]$. To this end, let $a >0$ be defined by \eqref{asupnI} as in Lemma \ref{coneCondImplnI} with $K=B_r(0)$, were $r>0$ is so large that 
\begin{align*}
\{ p^*Rm_{g_t}(x) \mid x \in M \ \text{and} \ p \in O^{g_t}_x  \} \subseteq B_r(0)
\end{align*}
for all $t \in [0,T_1]$, and set 
\begin{align*}
L &:= \max_{(x,t) \in M \times [0,T_1]} \| Rm_{g_t}(x) + Rm_{g_t}(x) \# I_{g_t}(x) \|_{g_t} \\
\text{and} \ \ P &:=  \max_{(x,t) \in M \times [0,T_1]} \| \lambda I_{g_t}(x) - Rm_{g_t}(x)\|_{g_t},
\end{align*}
where $I_{g_t} \in \Gamma(M,S^2_B(\Lambda^2T^*M))$ with $I_{g_t}(x)$ being the identity in $S^2_B(\Lambda^2T^*_xM)$ with respect to $g_t$ for all $x \in M$.
By the compactness of $M \times [0,T_1]$, the constants $r$, $L$ and $P$ are finite. Moreover, we choose $b > \frac{2|\lambda|L}{a}$ and $\varepsilon_0 \in (0,1)$ such that
\begin{align*}
 \varepsilon_0 e^{bT_1} < \min \left\{ \frac{1}{2},  \frac{ab-2|\lambda|L}{\lambda^2 \sqrt{2n(n-1)^3} +bP} \right\}.
\end{align*}
For $\varepsilon \in( 0,\varepsilon_0)$, we define  
\begin{align*}
R^{\varepsilon}_t := Rm_{g_t} + \varepsilon e^{bt} (\lambda I_{g_t} - Rm_{g_t}) = (1-\varepsilon e^{bt}) Rm_{g_t} + \varepsilon e^{bt} \lambda I_{g_t}
\end{align*}
for $t \in [0,T)$. Using Lemma \ref{evolutionRmallg} and that $I + I^{\#} = (n-1)I$, by \cite[Lemma 2.1]{BoehmWilkingAnnals}, we find that
\begin{align} \begin{split} \label{evolReps}
\nabla_{\!\frac{\partial}{\partial t}} R^{\varepsilon}_t &\stackrel{\ref{evolutionRmallg}}{=} -\varepsilon b e^{bt}Rm_{g_t} + (1-\varepsilon e^{bt}) \left( \Delta_{g_t}Rm_{g_t} + Rm_{g_t}^2 + Rm_{g_t}^{\#} \right) + \varepsilon b e^{bt} \lambda I_{g_t} \\
&\ = \ \Delta_{g_t}R^{\varepsilon}_t + \varepsilon b e^{bt} \left(\lambda I_{g_t} -Rm_{g_t} \right)  \\
&\ \ \ \ \ +  \frac{1}{1-\varepsilon e^{bt}} \big( (R^{\varepsilon}_t)^2 + (R^{\varepsilon}_t)^{\#} - \varepsilon^2 e^{2bt} \lambda^2 (I_{g_t} + I_{g_t}^{\#}) \\
& \ \ \ \ \ \ \  \ \ \ \ \ \ \  \ \ \ \ \ \ \  - 2 (1-\varepsilon e^{bt}) \varepsilon e^{bt} \lambda (Rm_{g_t} + Rm_{g_t} \# I_{g_t}) \big)  \\
&\ = \ \Delta_{g_t}R^{\varepsilon}_t + \varepsilon b e^{bt} \left(\lambda I_{g_t} -R^{\varepsilon}_t \right)+ \varepsilon^2 b e^{2bt} \left(\lambda I_{g_t} -Rm_{g_t} \right) \\
&\ \ \ \ \ + \frac{1}{1-\varepsilon e^{bt}} \left( (R^{\varepsilon}_t)^2 + (R^{\varepsilon}_t)^{\#} \right) - \frac{\varepsilon^2e^{2bt} \lambda^2(n-1)}{1- \varepsilon e^{bt}}I_{g_t} - 2\varepsilon e^{bt} \lambda (Rm_{g_t} + Rm_{g_t} \# I_{g_t}) .
\end{split}
\end{align}
By assumption and Lemma \ref{concondFolgtInterior}, $R^{\varepsilon}_0 = Rm_{g_0} + \varepsilon(\lambda I_{g_0} - Rm_{g_0} )$ is in the interior of $\Omega^{g_0}$.
We claim that $R^{\varepsilon}_t$ is in the interior of $\Omega^{g_t}$ for all $t \in [0,T_1]$. Suppose this is not true. Then there is a minimal time $t_0 \in (0, T_1]$ such that $R^{\varepsilon}_{t_0}(x_0) \in \partial \Omega^{g_{t_0}}_{x_0}$ for some $x_0 \in M$, since $M$ is compact. Hence, $R^{\varepsilon}_t(x) \in \Omega^{g_t}_x$ for all $t \in [0,t_0]$ and $x \in M$, and in particular $R^{\varepsilon}_{t_0} \in \Gamma(M,\Omega^{g_{t_0}})$.  
Let  $t \mapsto p_t \in O^{g_{t}}_{x_0}$ be parallel and set $S_0:=p_{t_0}^*R^{\varepsilon}_{t_0}(x_0) \in \partial \Omega$. Furthermore, for $i=1, \dots, n$, set $b_i := p_{t_0}(e_i)$, where $(e_1, \dots, e_n)$ denotes the standard basis of $\R^n$. Then $(b_1, \dots, b_n)$ is an orthonormal basis of $T_{x_0}M$ with respect to $g_{t_0}$.
We choose $\delta > 0$ such that
\begin{align*}
 \delta \sum\limits_{i=1}^n \|\nabla^{g_{t_0}}_{b_i}R^{\varepsilon}_{t_0}\|_{g_{t_0}}^2  < \varepsilon e^{bt_0} \left(ab-2|\lambda|L-\varepsilon_0 e^{bT_1}\lambda^2\sqrt{2n(n-1)^3} + \varepsilon_0 be^{bT_1}P \right) .
\end{align*}
By the choice of the constants $b$ and $\epsilon_0$, the right-hand side of this inequality is positive.
Since $\Omega$ is Bianchi-convex, there is a supporting submanifold $N$ of $\Omega$ in $S_0$ with
\begin{align*} 
\sum_{i=1}^n \text{\Gemini}^N_S(T_i,T_i) \leq \delta \sum_{i=1}^n \|T_i\|^2
\end{align*}
for all $S \in N$ and $(T_1, \dots, T_n) \in (T_SN)^n$ that satisfy the second Bianchi identity. It follows that $(p_{t_0}^{-1})^*N$ is a supporting submanifold of $\Omega^{g_{t_0}}_{x_0}$ in $R^{\varepsilon}_{t_0}(x_0)$ with 
\begin{align*}
\sum_{i=1}^n \text{\Gemini}^{(p_{t_0}^{-1})^*N}_S(T_i,T_i) \leq \delta \sum_{i=1}^n \|T_i\|_{g_{t_0}}^2
\end{align*}
for all $S \in (p_{t_0}^{-1})^*N$ and $(T_1, \dots, T_n) \in (T_S(p_{t_0}^{-1})^*N)^n$ that satisfy the second Bianchi identity.

By $\mathbf{n}_{S_0}$, we denote the unit normal on $N$ at $S_0$ pointing in the opposite direction of $\Omega$. Then $\mathbf{n}_{R^{\varepsilon}_{t_0}(x_0)} := (p_{t_0}^{-1})^*\mathbf{n}_{S_0}$ is the unit normal on $(p_{t_0}^{-1})^*N$ at $R^{\varepsilon}_{t_0}(x_0)$ pointing in the opposite direction of $\Omega^{g_{t_0}}_{x_0}$. 

Let further $r^N$ be a signed distance function from $N$. Now, using that $\Omega$ is invariant under the ordinary differential equation \eqref{ODEeinleitung} in combination with Proposition \ref{ODETangkegel} and Lemma \ref{tangentconeinhalfspacessm}, and applying Lemma \ref{laplacezeigtnachinnen}, we can compute that
\begin{align*}
\frac{d}{dt}\bigg{|}_{t=t_0}r^N(p_t^*R^{\varepsilon}_t(x_0)) &= \ dr^N_{S_0}\left( \frac{d}{dt}\bigg{|}_{t=t_0}p_t^*R^{\varepsilon}_t(x_0) \right) \\
&= \ dr^N_{S_0}\left( p_{t_0}^*\nabla_{\!\frac{\partial}{\partial t}}R^{\varepsilon}_{t}(x_0) \Big{|}_{t=t_0} \right) \\
&\!\stackrel{\eqref{evolReps}}{=} \left\langle \mathbf{n}_{S_0} , p_{t_0}^*\Delta_{g_{t_0}}R^{\varepsilon}_{{t_0}}(x_0) \right\rangle + \frac{1}{1- \varepsilon e^{bt_0}} \left\langle  \mathbf{n}_{S_0}, p_{t_0}^*\big(R^{\varepsilon}_{{t_0}}(x_0)^2 + R^{\varepsilon}_{{t_0}}(x_0)^{\#}\big) \right\rangle \\
& \ \ \ \ -2\varepsilon e^{bt_0} \lambda \left\langle \mathbf{n}_{S_0}, p_{t_0}^*\left(Rm_{g_{t_0}}(x_0) + Rm_{g_{t_0}}(x_0) \# I_{g_{t_0}}(x_0)\right) \right\rangle \\
& \ \ \ \ - \frac{\varepsilon^2 e^{2bt_0}\lambda^2(n-1)}{1- \varepsilon e^{bt_0}} \left\langle \mathbf{n}_{S_0} , I \right\rangle + \varepsilon b e^{bt_0} \left\langle \mathbf{n}_{S_0}, p_{t_0}^*\left( \lambda I_{g_{t_0}}(x_0) - R^{\varepsilon}_{t_0}(x_0) \right) \right\rangle \\
& \ \ \ \ + \varepsilon^2 b e^{2bt_0} \left\langle \mathbf{n}_{S_0}, p_{t_0}^*\left( \lambda I_{g_{t_0}}(x_0) - Rm_{g_{t_0}}(x_0) \right) \right\rangle 
\end{align*}

\begin{align*}
&\leq \ \left\langle (\mathbf{n}_{R^{\varepsilon}_{t_0}(x_0)} , \Delta_{g_{t_0}}R^{\varepsilon}_{{t_0}}(x_0) \right\rangle_{g_{t_0}} +2 \underbrace{ \left\langle  \mathbf{n}_{S_0}, S_0^2 + S_0^{\#} \right\rangle}_{\stackrel{\ref{ODETangkegel},\ref{tangentconeinhalfspacessm}}{\leq} 0}  + \frac{\varepsilon^2 e^{2bt_0}\lambda^2(n-1)}{1-\varepsilon e^{bt_0}} \|I\|\\
& \ \ \ \ -2\varepsilon e^{bt_0} \lambda \left\langle \mathbf{n}_{R^{\varepsilon}_{t_0}(x_0)}, \left(Rm_{g_{t_0}}(x_0) + Rm_{g_{t_0}}(x_0) \# I_{g_{t_0}}(x_0)\right) \right\rangle_{g_{t_0}}  \\
& \ \ \ \ + \varepsilon b e^{bt_0} \underbrace{\left\langle \mathbf{n}_{R^{\varepsilon}_{t_0}(x_0)}, \left( \lambda I_{g_{t_0}}(x_0) - R^{\varepsilon}_{t_0}(x_0) \right) \right\rangle_{g_{t_0}}}_{\leq -a}  \\
& \ \ \ \ + \varepsilon^2 b e^{2bt_0} \left\langle \mathbf{n}_{R^{\varepsilon}_{t_0}(x_0)}, \left( \lambda I_{g_{t_0}}(x_0) - Rm_{g_{t_0}}(x_0) \right) \right\rangle_{g_{t_0}} \\
&\leq \ \left\langle \mathbf{n}_{R^{\varepsilon}_{t_0}(x_0)} , \Delta_{g_{t_0}}R^{\varepsilon}_{{t_0}}(x_0) \right\rangle_{g_{t_0}} +  \varepsilon^2 e^{2bt_0} \lambda^2 \sqrt{2n(n-1)^3} \\
& \ \ \ \ + 2\varepsilon e^{bt_0} |\lambda| \underbrace{\left\|Rm_{g_{t_0}}(x_0) + Rm_{g_{t_0}}(x_0) \# I_{g_{t_0}}(x_0)\right\|_{g_{t_0}}}_{\leq L} \\
&\ \ \ \ - \varepsilon b e^{bt_0} a + \varepsilon^2 b e^{2bt_0} \underbrace{\|\lambda I_{g_{t_0}}(x_0) - R^{\varepsilon}_{t_0}(x_0)\|_{g_{t_0}}}_{\leq P}\\
&\stackrel{\ref{laplacezeigtnachinnen}}{\leq} \; \delta \sum_{i=1}^n \|\nabla^{g_{t_0}}_{b_i}R^{\varepsilon}_{{t_0}}\|_{g_{t_0}}^2  + \varepsilon  e^{bt_0} \big(2|\lambda|L -  a b + \varepsilon_0 e^{bT_1} \lambda^2 \sqrt{2n(n-1)^3} + \varepsilon_0 e^{bT_1}P\big) \\
&< \ 0.
\end{align*}
Here, we could apply Lemma \ref{laplacezeigtnachinnen} for $C = \Omega^{g_{t_0}}$ since with $Rm_{g_{t_0}}$ also $R^{\varepsilon}_{t_0} = Rm_{g_{t_0}} + \varepsilon e^{bt_0} ( \lambda I_{g_{t_0}}- Rm_{g_{t_0}})$ satisfies the second Bianchi identity and, in addition, $R^{\varepsilon}_{t_0} \in \Gamma(M,\Omega^{g_{t_0}})$. This together with the fact that $r^N(p_{t_0}^*R^{\varepsilon}_{t_0}(x_0)) = r^N(S_0) = 0$ yields that there is a $\mu >0$ such that 
\begin{align*}
r^N(p_t^*R^{\varepsilon}_t(x_0)) > 0
\end{align*}
for all $t \in (t_0-\mu, t_0)$. Therefore, $p_t^*R^{\varepsilon}_t(x_0) \notin \Omega$, thus $R^{\varepsilon}_t(x_0) \notin \Omega^{g_t}_{x_0}$ for all $t \in (t_0-\mu, t_0)$. This, however, is a contradiction to $R^{\varepsilon}_t(x_0) \in \Omega^{g_t}_{x_0}$ for all $t \in [0,t_0]$ as assumed in the beginning of the proof. Hence, as we claimed, $R^{\varepsilon}_t(x)$ is in the interior of $\Omega^{g_t}_x$ for all $x \in M$ and $t \in [0,T_1]$.
\medskip

Since $\Omega^{g_t}$ is closed and $R^{\varepsilon}_t(x)$ converges to $Rm_{g_t}(x)$ as $\varepsilon$ tends to zero for each $t \in [0,T_1]$ and $x \in M$, we have that $Rm_{g_t}(x) \in \Omega^{g_t}_x$ for all $t \in [0,T_1]$ and $x \in M$. Since $T_1$ was chosen arbitrarily, this is true for all $t \in [0,T)$. Consequently, we have shown that $\Omega$ is invariant under the Ricci flow.
\end{proof}

\bibliography{Literatur}

\end{document}